\documentclass[sn-basic,Numbered]{sn-jnl}

\usepackage{graphicx}%
\usepackage{multirow}%
\usepackage{amsmath,amssymb,amsfonts}%
\usepackage{amsthm}%
\usepackage{mathrsfs}%
\usepackage[title]{appendix}%
\usepackage{xcolor}%
\usepackage{textcomp}%
\usepackage{manyfoot}%
\usepackage{booktabs}%
\usepackage{algorithm}%
\usepackage{algorithmicx}%
\usepackage{algpseudocode}%
\usepackage{listings}%
\usepackage{makecell}

\theoremstyle{thmstyleone}%
\newtheorem{theorem}{Theorem}
\newtheorem{proposition}[theorem]{Proposition}%

\newtheorem{fact}{Fact}
\newtheorem{condition}{Condition}
\newtheorem{lemma}{Lemma}

\theoremstyle{thmstyletwo}%
\newtheorem{remark}{Remark}%

\theoremstyle{thmstylethree}%
\newtheorem{definition}{Definition}%

\begin{document}

\title{Linear Convergence of the Proximal Gradient Method for Composite Optimization Under the Polyak-{\L}ojasiewicz Inequality and Its Variant}


\author[1]{\fnm{Qingyuan} \sur{Kong}}

\author*[1]{\fnm{Rujun} \sur{Jiang}}\email{rjjiang@fudan.edu.cn}

\author[1]{\fnm{Yihan} \sur{He}}

\affil[1]{\orgdiv{School of Data Science}, \orgname{Fudan University}, \orgaddress{\postcode{200433}, \state{Shanghai}, \country{China}}}






\abstract{We study the linear convergence rates of the proximal gradient method for composite functions satisfying two classes of Polyak-{\L}ojasiewicz (PL) inequality: the PL inequality, the variant of PL inequality defined by the proximal map-based residual. Using the performance estimation problem, we either provide new explicit linear convergence rates or improve existing complexity bounds for minimizing composite functions under the two classes of PL inequality. Finally, we illustrate numerically the effects of our theoretical results.}




\maketitle

\section{Introduction}\label{sec1}

In this paper, we consider the following composite minimization problem
\begin{equation} \label{pb:pgm}
\min_{x\in\mathbb{R}^d} F(x):=f(x)+h(x),
\end{equation}
where $f: \mathbb{R}^d \to\mathbb{R}$ is a function with Lipschitz continuous gradient, and $h: \mathbb{R}^d \to\mathbb{R} \cup \{+\infty\}$ is a closed, proper and convex function. In addition, we assume that the optimal set $X_{*}(F) := \operatorname*{argmin}_{x\in\mathbb{R}^d} F(x)$ is nonempty.
We denote the class of closed and proper convex functions by $\mathcal{F}_{0,\infty}(\mathbb{R}^{d})$, the class of $L$-smooth convex functions by $\mathcal{F}_{0, L}(\mathbb{R}^{d})$, and the class of $L$-smooth functions by $\mathcal{F}_{-L, L}(\mathbb{R}^{d})$.
For simplicity, we say that $F$ is a nonconvex composite function if $f\in\mathcal{F}_{-L, L}(\mathbb{R}^{d})$ and a convex composite function if $f\in\mathcal{F}_{0, L}(\mathbb{R}^{d})$.

The composite optimization problems arise in numerous applications, including compressed sensing \cite{donoho2006compressed}, signal processing \cite{combettes2011proximal}, and machine learning \cite{mosci2010solving}. In recent years, the proximal gradient method (PGM) is a widely used method for composite optimization \cite{beck2017first}.
Define the proximal operator of $h$ at $x \in \mathbb{R}^d$ as
\[
\mathrm{prox}_{h}(x):={\operatorname*{argmin}}_{y\in\mathbb{R}^{d}}\:\left\{h(y)+\frac{1}{2}\|y-x\|^{2}\right\}.
\]
We summarise the PGM in Algorithm \ref{alg:the PGM}.
\begin{algorithm}[H] \label{alg:the PGM}
\caption{Proximal gradient method (PGM)}
\begin{algorithmic}[1]
\Require $x_{1}\in\mathbb{R}^{d}$, $\gamma \in (0,\frac{2}{L})$.
\For {$i=1,\ldots,N$}
\State $x_{i+1}=\mathrm{prox}_{\gamma h}\left(x_i-\gamma\nabla f(x_i)\right)$
\EndFor
\end{algorithmic}
\end{algorithm}

It is well-known that for composite functions with strongly convex $f$, the PGM achieves a linear convergence rate \cite{beck2017first}.
Relaxing the strong convexity condition while retaining a linear convergence rate of the PGM has been extensively studied in the literature.
The PL inequality was initially introduced by Polyak \cite{polyak1963gradient}, which is used to show linear convergence of gradient descent. 
Bolte et al. \cite{bolte2017error} derived a linear convergence rate of the PGM under the Kurdyka-{\L}ojasiewicz (KL) inequality with an exponent of $\frac{1}{2}$ for convex composite functions. When the KL exponent is $\frac{1}{2}$, the KL inequality is known as PL inequality for nonsmooth functions.
Li and Pong \cite{li2018calculus} studied the calculus of the exponent of KL inequality and proved the linear convergence for the iterates of the PGM when the PL inequality holds.
Recently, Garrigos et al. \cite{garrigos2023convergence} provided an \emph{explicit} linear convergence rate of the PGM under the PL inequality. 
Zhang and Zhang \cite{zhang2019new} used a proximal map-based
residual to define the generalized PL inequality for convex composite functions, and showed an explicit linear convergence rate of the PGM under this inequality. 

Drori and Teboulle \cite{drori2014performance} first proposed the performance estimation problem (PEP) for the worst-case analysis of first-order methods, which is based on using a semidefinite programming  (SDP) relaxation.
Subsequently, PEP has been widely used for analyzing the worst-case convergence rates of various first-order methods, such as the gradient methods or their accelerated versions \cite{drori2014performance,de2017worst}, the PGM \cite{taylor2018exact}, the alternating direction method
of multipliers \cite{zamani2023exact}, and the difference-of-convex algorithm \cite{abbaszadehpeivasti2024rate}.
While the original SDP relaxation in \cite{drori2014performance} is based on a discrete relaxation of the PEP, Taylor et. al. \cite{taylor2017smooth, taylor2017exact} proposed an exact reformulation of the PEP by deriving the sufficient and necessary interpolation conditions for the function class $\mathcal F_{\mu, L}(\mathbb{R}^d)$ with $\mu\ge 0$ and $L\in(\mu,+\infty]$. 
For a composite function with a strongly convex $f$, Taylor et al. \cite{taylor2018exact} showed the exact worst-case linear convergence rates of the PGM in terms of objective function accuracy, distance to optimality, and residual gradient norm.

The main contribution of this paper is applying performance estimation to analyze the convergence of the PGM for composite functions under two classes of PL inequalities: the PL inequality, the variant of PL inequality defined by the proximal map-based residual, which we call RPL for simplicity. When the PL inequality holds, we provide the first \emph{explicit} linear convergence rate for nonconvex composite functions, and improve the existing convergence results in \cite{garrigos2023convergence} for convex composite functions. When the RPL inequality holds, we also provide the first \emph{explicit} linear convergence rate for nonconvex composite functions, and find a better bound than that given in \cite{zhang2019new} for convex composite functions. 
Furthermore, for the two classes of PL inequalities, we deduce the ``optimal'' constant step size based on the convergence rates obtained above.
We compare in Table \ref{table} the explicit worst-case convergence rates we obtained with the existing results in the literature.

\begin{table}[h]
\footnotesize
\caption{Comparison of different classes of functions and PL inequalities.}\label{table}
\renewcommand{\arraystretch}{2}
\begin{tabular}{@{\extracolsep\fill}clll}
\hline
\multicolumn{1}{l}{Fun. class} & Con. & Convergence rate & Existing result \\ \hline
\multirow{2}{*}{\makecell{$f \in \mathcal{F}_{0,L}(\mathbb{R}^{d})$, \\ $h \in \mathcal{F}_{0,\infty}(\mathbb{R}^{d})$}} & PL & $\max\{\frac{1}{2 \gamma \mu + 1}, \frac{(L\gamma - 1)^2}{(L\gamma - 1)^2 - L \gamma^{2} \mu + 2 \gamma \mu}\}$ & $\frac{1}{1+\gamma\mu(2-\gamma L)}$ \cite{garrigos2023convergence}    \\ \cmidrule{2-4}
& RPL & $\max\{\frac{1-\gamma \mu}{1+ \gamma \mu},\frac{-2L\gamma^2\mu+L\gamma+3\gamma\mu-2}{L\gamma-\gamma\mu-2},\frac{(L\gamma - 1)^2}{(L\gamma - 1)^2 - L \gamma^{2} \mu + 2 \gamma \mu}\}$ & $\frac{1-\gamma \mu}{1+\gamma \mu}$($\gamma \in (0,\frac{1}{L})$) \cite{zhang2019new}   \\ \hline
\multirow{2}{*}{\makecell{$f \in \mathcal{F}_{-L,L}(\mathbb{R}^{d})$,\\ $h \in \mathcal{F}_{0,\infty}(\mathbb{R}^{d})$}}   & PL & $\max\{\frac{(L\gamma + 1)^2}{(L\gamma + 1)^2 + L \gamma^{2} \mu + 2 \gamma \mu},  \frac{(L\gamma - 1)^2}{(L\gamma - 1)^2 - L \gamma^{2} \mu + 2 \gamma \mu}\}$ & /   \\ \cmidrule{2-4}
& RPL & $\frac{L+\mu\left(L\gamma-1\right)^2-\mu}L$ & /   \\
 \hline
\end{tabular}
\footnotetext{
In the table, $L$ is the Lipschitz constant of $\nabla f(x)$, $\mu$ is the constant in their classes of PL inequalities, and $\gamma$ is the step size.}
\end{table}

The rest of the paper is organized as follows. In Sect. \ref{sect2}, we provide preliminary definitions of PL inequalities and some interpolation conditions used in the proof.
In Sect. \ref{sect3}, we present explicit linear convergence rates of the PGM and derive its ``optimal'' step size under the PL inequality. Subsequently, in Sect. \ref{sect4}, we extend our analysis to the RPL inequality setting and also establish new linear convergence results. In Sect. \ref{sect5}, we verify the performance of the ``optimal'' step size in practical problems through numerical experiments. We end the paper with some conclusions in Sect. \ref{sect6}.

\textbf{Notation.}
For the $d$-dimensional Euclidean space $\mathbb{R}^d$, the notation $\left<\cdot,\cdot\right>$ and $||\cdot||$ is the Euclidean inner product and norm, respectively. The distance function is denoted by $d(x, X):= \inf_{y \in X}\|x-y\|$ for any nonempty $X \subset \mathbb{R}^d$.

\section{Preliminaries} \label{sect2}

We consider the convergence for the PGM under the following two settings, where the first considers nonconvex smooth function $f$, and the latter considers convex smooth $f$.
\begin{condition}\label{con1}
The function $f \in \mathcal{F}_{-L, L}(\mathbb{R}^d)$ and $h \in \mathcal{F}_{0,\infty}(\mathbb{R}^d)$.
\end{condition}
\begin{condition}\label{con2}
The function $f \in \mathcal{F}_{0, L}(\mathbb{R}^d)$ and $h \in \mathcal{F}_{0,\infty}(\mathbb{R}^d)$.
\end{condition}

We say that a smooth function $f$ satisfies the PL inequality if the following inequality holds for some $\mu > 0$,
\begin{equation} \label{eq:PL}
||\nabla f(x)||^{2}\geq 2\mu(f(x)-f_{*}),\quad\forall x \in \mathbb{R}^d,
\end{equation}
where $f_*$ denotes the global optimal function value. The PL inequality implies that each stationary point is a global minimum.

Let $\partial F(x)$ denote the Clarke subdifferential of a local Lipschitz function $F$ at $x$  \cite{clarke1990optimization}. 
We say the function $F=f+h$ satisfies the PL inequality if there exists $\mu>0$ such that
\begin{equation} \label{eq:sPL}
d(0,\partial F(x))^2 \geq {2\mu}(F(x)-F_{*}),\quad \forall x \in \mathbb{R}^d,
\end{equation}
where $\partial F(x) = \nabla f(x) + \partial h(x)$ because $f$ is smooth \cite[Proposition 2.3.3]{clarke1990optimization}.
Note that the PL inequality here is a special case of the KL inequality with exponent $\frac{1}{2}$ \cite{bolte2017error}.

For composite function $F$, Li and Li \cite{li2018simple} used proximal map-based residual to define the RPL inequality:
\begin{equation} \label{eq:RPL}
\exists\mu>0,\text{ such that }\|\mathcal{G}_\gamma(x)\|^2\geq2\mu(F(x)-F_{*}), \quad \forall x \in \mathbb{R}^d,
\end{equation}
where the proximal map-based residual is defined by
\begin{equation} \label{eq:gradm}
\mathcal{G}_\gamma(x):=\frac{1}{\gamma}\Big(x-\mathrm{prox}_{\gamma h}\big(x-\gamma\nabla f(x)\big)\Big). \nonumber
\end{equation}

The following lemma provides the relationship between the proximal map-based residual and the distance of the subdifferential to 0, which is helpful in our analysis.
\begin{lemma}\cite[Theorem 3.5]{drusvyatskiy2018error} \label{lemma}
Consider the composite function $F = f + h$. For any $x \in \mathbb{R}^d$, it holds that
$
\|\mathcal{G}_{\gamma}(x)\| \leq d(0,\partial F(x)). \nonumber
$
\end{lemma}
By Lemma~\ref{lemma}, it is clear that the RPL inequality \eqref{eq:RPL} implies the PL inequality \eqref{eq:sPL}, since
\begin{equation}
F(x) - F_* \leq \frac{1}{2\mu}\|\mathcal{G}_{\gamma}(x)\|^2 \leq \frac{1}{2\mu} d(0,\partial F(x))^2. \nonumber
\end{equation}
However, it appears challenging to derive the converse implication from the PL inequality to the RPL inequality.

Our analysis heavily depends on the PEP technique, which can be summarised as the following procedure:
(1) The abstract PEP problem can be discretized as a finite-dimensional optimization problem, where interpolation conditions should be used if available.
(2) By introducing the Gram matrix, the finite-dimensional optimization problem is relaxed as a convex SDP problem, which is solvable via various off-the-shelf solvers.
(3) Optimal dual solutions are first obtained numerically, and then the analytical expressions are guessed and verified analytically.
(4) Finally,  the worst-case convergence rate is obtained by using nonnegative linear combinations of constraints in the discretized formulation of the PEP problem.

We denote first-order oracle at a given point $x$ for function $F$ by $\mathcal{O}_F(x)=\{x,F(x),\tilde{\nabla}F(x)\}$, where $\tilde{\nabla}F(x) \in \partial F(x)$ is a Clarke subgradient of $F$ at $x$ (or gradient if $F$ is differentiable at $x$). We consider a method $\mathcal{M}$ that executes $N$ steps of iterations $\{x_i\}_{1 \leq i \leq N}$ starting from the initial iterate $x_1$, where the $i+1$-th iteration generated by $\mathcal{M}$ can be written as $x_{i+1} = \mathcal{M}_{i+1}(x_1, \mathcal{O}_F(x_1), \ldots,\mathcal{O}_F(x_i))$. Let $I = \{1, \ldots, N, *\}$ be an index set containing an optimal solution $x_*$ of the problem. At each iterate $x_i$, we denote the output of the oracle by $\{x_i, g_i, f_i\}$, where $f_i = F(x_i)$ and $g_i \in \partial F(x_i)$.
Based on the above concepts, we introduce the definition of $\mathcal{F}$-interpolation. 
\begin{definition}\cite[Definition 2]{taylor2017smooth}
Let $I = \{1, \ldots, N, *\}$ be an index set, and consider the set of triples $\{(x_i, g_i, f_i)\}_{i \in I}$, where $x_i\in \mathbb{R}^d$, $g_i\in \mathbb{R}^d$ and $f_i\in \mathbb{R}$, $\forall i \in I$. The set of triples $\{(x_i, g_i, f_i)\}_{i \in I}$ is $\mathcal{F}$-interpolable if and only if there exists a function $F \in \mathcal{F}$ such that for all $i \in I$, we have both $f_i=F(x_i)$ and $g_i \in \partial F(x_i)$.
\end{definition}
Recall the following interpolation conditions, which play a key role in our analysis.

\begin{fact}\cite[Theorem 3.10]{taylor2017exact} \label{fact1}
The set $\{(x_i;g_i;f_i)\}_{i\in I}$ is  $\mathcal{F}_{-L, L}$-interpolable if and only if the following inequality holds $\forall i,j\in I$:
\begin{equation} \label{subeq:L-interpo}
A_{i,j}=f_{j}-f_{i}-\frac{L}{4}\|x_{i}-x_{j}\|^{2}+\frac{1}{2}\langle g_{i}+g_{j},x_{j}-x_{i}\rangle+\frac{1}{4L}\|g_{i}-g_{j}\|^{2}\le 0.
\end{equation}
\end{fact}

\begin{fact}\cite[Corollary 1]{taylor2017smooth} \label{fact2}
The set $\{(x_i;g_i;f_i)\}_{i\in I}$ is  $\mathcal{F}_{0, L}$-interpolable if and only if the following inequality holds $\forall i,j\in I$:
 \begin{equation} \label{eq:Lconv}
B_{i,j} := f_{j} - f_{i} + \langle g_{j},x_{i}-x_{j}\rangle + \frac{1}{2L}\|g_{i}-g_{j}\|^{2} \leq 0.
\end{equation}
\end{fact}
\begin{fact}\cite[Theorem 1]{taylor2017smooth} \label{fact3}
The set $\{(x_i;s_i;h_i)\}_{i\in I}$ is $\mathcal{F}_{0, \infty}$-interpolable if and only if the following set of conditions holds $\forall i,j\in I$:
 \begin{equation} \label{subeq:cov-interpo}
C_{i,j}:= h_{j}-h_{i}+\langle s_{j},x_{i}-x_{j}\rangle \leq 0.
\end{equation}
\end{fact}

\section{Linear Convergence under the PL Inequality} \label{sect3}
In this section, we investigate the linear convergence rate of the PGM under the PL inequality.
Furthermore, we obtain the explicit linear convergence rate and the ``optimal'' step size and compare them with the existing results.

\subsection{Nonconvex Composite Case} \label{sect3.1}
We now consider the performance of the PGM under Condition \ref{con1} and assume the PL inequality \eqref{eq:sPL} holds.
In the following of this paper, we set the performance criterion as the objective function accuracy $F(x_2)-F_*$. In this setting, the PEP for the PGM can be written as follows:
\begin{equation} \label{pb:pep}
\begin{aligned}
\max \quad & F(x_2)-F_* \\
\text{s.t.} \quad & x_2\text{ is generated by Algorithm \ref{alg:the PGM} w.r.t.}F,x_1, \\
& \text{$f \in \mathcal{F}_{-L,L}(\mathbb{R}^d)$, $h \in \mathcal{F}_{0,\infty}(\mathbb{R}^d)$,}  \\
& F(x)-F_* \leq \frac{1}{2\mu} d(0,\partial F(x))^2, \forall x\in X, \\
& F(x_1) - F_* \leq \Delta, \\
& F(x) \geq F_*, \forall x \in \mathbb{R}^d, \\
& x_1\in\mathbb{R}^d,
\end{aligned}
\end{equation}
where $X \subseteq \mathbb{R}^d$, $\Delta \geq 0$ denotes the difference between the optimal function value $F_*$ and the value of $F$ at the starting point. In addition, $F$ and $x_1$ are decision variables. This is an infinite-dimensional optimization problem with an infinite number of constraints, and consequently intractable in general.
In Section \ref{sect2}, Facts \ref{fact1} and \ref{fact3} provide necessary and sufficient conditions for the interpolation of $L$-smooth functions and closed, proper and convex functions. We can discretize problem \eqref{pb:pep} as a finite-dimensional optimization problem by using the following (in)equalities and interpolation conditions \eqref{subeq:L-interpo} and \eqref{subeq:cov-interpo} in Facts \ref{fact1} and \ref{fact3}.

First, following \cite{taylor2018exact}, the iteration $x_{i+1}=\mathrm{prox}_{\gamma h}\left(x_{i}-\gamma\nabla f(x_{i})\right)$ can be rewritten using necessary and sufficient optimality conditions on the definition of the proximal operator:
\begin{equation} \label{eq:iter}
x_{i+1}=x_i-\gamma(g_i+s_{i+1})
\end{equation}
for some $s_{i+1}\in\partial h(x_{i+1})$, and $g_i = \nabla f(x_i)$. Note that $g_*+s_* = 0$.

When the function $F$ satisfies the PL inequality \eqref{eq:sPL}, we directly discretize the PL inequality \eqref{eq:sPL} and pick a subgradient $s_i \in \partial h(x_i)$.  Then we obtain
\begin{equation} \label{eq:PLdiscrete}
D_i:=f_i+h_i-F_*- \frac{1}{2\mu} \| g_i+s_i\|^2  \leq 0.
\end{equation}

We then use the PEP procedure outlined in Section \ref{sect2}, where problem \eqref{pb:pep} is discretized as a finite-dimensional optimization problem using the above (in)equalities.
We then provide an explicit convergence rate of the PGM under the PL inequality \eqref{eq:sPL}.

\begin{theorem} \label{thm:noncov_sPL}
Suppose that Condition \ref{con1} holds, $F$ satisfies the PL inequality \eqref{eq:sPL} on $X=\{x:F(x)\leq F(x_1)\}$, and $x_2$ is generated by Algorithm \ref{alg:the PGM}. Then the following holds.
\begin{enumerate}
  \item If $\gamma \in\left(0,\frac{\sqrt{3}}{L}\right]$, then
$
F(x_2) - F_* \leq \frac{(L\gamma + 1)^2}{(L\gamma + 1)^2 + L \gamma^{2} \mu + 2 \gamma \mu} (F(x_1) - F_*). $
  \item If $\gamma \in\left(\frac{\sqrt{3}}{L},\frac{2}{L}\right)$, then
$
F(x_2) - F_* \leq \frac{(L\gamma - 1)^2}{(L\gamma - 1)^2 - L \gamma^{2} \mu + 2 \gamma \mu} (F(x_1) - F_*). $
\end{enumerate}
\end{theorem}
\begin{proof}1. First consider the case of $\gamma \in\left(0,\frac{\sqrt{3}}{L}\right]$.
Let
$\alpha_1 = \frac{L^{2} \gamma^{2} + 3 L \gamma + 2}{(L\gamma + 1)^2 + L \gamma^{2} \mu + 2 \gamma \mu}, \alpha_2 = \frac{L \gamma + 1}{(L\gamma + 1)^2 + L \gamma^{2} \mu + 2 \gamma \mu},$
$\beta_1  =  \frac{(L\gamma + 1)^2}{(L\gamma + 1)^2 + L \gamma^{2} \mu + 2 \gamma \mu}$, and $\lambda_1 = \frac{L \gamma^{2} \mu + 2 \gamma \mu}{(L\gamma + 1)^2 + L \gamma^{2} \mu + 2 \gamma \mu}.$
It is obvious that $\alpha_1, \alpha_2, \beta_1$, and $\lambda_1$ are nonnegative. We obtain that
\begin{equation}
\begin{aligned}
F(x_2) & - F_* - (\alpha_1-\alpha_2) (F(x_1) - F_*)-\alpha_1 A_{1,2} -\alpha_2 A_{2,1} - \beta_1 C_{1,2} -\lambda_1 D_{2} \\
& = \frac{L^2\gamma^2-3}{4L\left(L\gamma^2\mu + 2\gamma\mu + \left(L\gamma+1\right)^2\right)} \|Lg_1\gamma+L\gamma s_2+g_1-g_2\|^2\leq 0,
\end{aligned} \nonumber
\end{equation}
where $A_{1,2}$ and $A_{2,1}$ are defined in \eqref{subeq:L-interpo}, $C_{1,2}$ is defined in \eqref{subeq:cov-interpo}, 
$D_2$ is defined in \eqref{eq:PLdiscrete},
and the last inequality holds since $\gamma \in (0,\frac{\sqrt{3}}{L}]$.
Because the terms after the nonnegative parameters $\alpha_1,\alpha_2,\beta_1,\lambda_1$ are nonpositive, we have
$
F(x_2) - F_* \leq (\alpha_1-\alpha_2)(F(x_1) - F_*).
$
We then obtain the desired result by substituting the values of $\alpha_1$ and $\alpha_2$.

2. For the case of $\gamma \in\left(\frac{\sqrt{3}}{L},\frac{2}{L}\right)$, let
$\alpha_1 = \frac{L \gamma - 1}{(L\gamma - 1)^2 - L \gamma^{2} \mu + 2 \gamma \mu}, \alpha_2 = \frac{- L^{2} \gamma^{2} + 3 L \gamma - 2}{(L\gamma - 1)^2 - L \gamma^{2} \mu + 2 \gamma \mu},$
$\beta_1 = \frac{(L\gamma - 1)^2}{(L\gamma - 1)^2 - L \gamma^{2} \mu + 2 \gamma \mu}, \lambda_1 = \frac{- L \gamma^{2} \mu + 2 \gamma \mu}{(L\gamma - 1)^2 - L \gamma^{2} \mu + 2 \gamma \mu}.$
It is obvious that $\alpha_1, \alpha_2, \beta_1$, and $\lambda_1$ are nonnegative. By multiplying the inequalities by their respective factors and then summing them, we obtain that
\begin{equation}
\begin{aligned}
F(x_2) & - F_* - (\alpha_1-\alpha_2) (F(x_1) - F_*)-\alpha_1 A_{1,2}-\alpha_2 A_{2,1}-\beta_1 C_{1,2}-\lambda_1 D_2 \\
& = \frac{-L^2\gamma^2+3}{4L\left(-L\gamma^2\mu+2\gamma\mu+\left(L\gamma-1\right)^2\right)}\|L\gamma g_1+L\gamma s_2-g_1+g_2\|^2 \leq 0,
\end{aligned} \nonumber
\end{equation}
where $A_{1,2}$ and $A_{2,1}$ are defined in \eqref{subeq:L-interpo}, $C_{1,2}$ is defined in \eqref{subeq:cov-interpo}, 
$D_2$ is defined in \eqref{eq:PLdiscrete}, and the last inequality holds since $\gamma \in (\frac{\sqrt{3}}{L},\frac{2}{L})$. The remainder of the proof follows a similar approach to the first case.
\end{proof}

We remark that the parameters $\alpha_1, \alpha_2, \beta_1$, and $\lambda_1$ are obtained from the SDP relaxation of the PEP problem, which is derived based on investigating the dual optimal solutions from numerous numerical experiments. Note that the explicit dual feasible solution and convergence rate in Theorem \ref{thm:noncov_sPL} match those of the SDP relaxation from numerical experiments.

\begin{remark}
Existing literature has studied the linear convergence properties of gradient methods for nonconvex and nonsmooth functions under the PL inequality \eqref{eq:sPL}, e.g, Li and Pong \cite{li2018calculus} showed that if $F$ satisfies the PL inequality \eqref{eq:sPL}, the iterative sequence $\{x_k\}$ converges locally linearly to a stationary point of $F$. However, to the best of our knowledge, there is no \emph{explicit} rate of linear convergence in existing literature. In Theorem \ref{thm:noncov_sPL}, we provide the first \emph{explicit} linear convergence rate of the PGM on objective function values for convex composition minimization.
\end{remark}

The following proposition provides the ``optimal'' step size in the sense of minimizing the worst-case convergence rate in Theorem \ref{thm:noncov_sPL}.
\begin{proposition}\label{prop:1}
Suppose that $F$ satisfies the same conditions in Theorem \ref{thm:noncov_sPL}. Then the ``optimal" step size for the PGM with the worst-case convergence rate in Theorem \ref{thm:noncov_sPL} is given by $\frac{\sqrt{3}}{L}$ in $\gamma \in (0,\frac{2}{L})$.
\end{proposition}
\begin{proof}
Let
$
\tau(\gamma) =\begin{cases}\frac{(L\gamma + 1)^2}{(L\gamma + 1)^2 + L \gamma^{2} \mu + 2 \gamma \mu}&\gamma \in\left(0,\frac{\sqrt{3}}{L}\right]\\
\frac{(L\gamma - 1)^2}{(L\gamma - 1)^2 - L \gamma^{2} \mu + 2 \gamma \mu}&\gamma \in\left(\frac{\sqrt{3}}{L},\frac{2}{L}\right).\end{cases} \nonumber
$
Note that $\tau'(\gamma) < 0$ with $\gamma \in\left(0,\frac{\sqrt{3}}{L}\right)$, and that $\tau'(\gamma) > 0$ with $\gamma \in\left(\frac{\sqrt{3}}{L},\frac{2}{L}\right)$. Note also that $\tau(\gamma)$ is continuous at $\frac{\sqrt3}{L}$.
Hence, $\tau(\gamma)$ is decreasing at $\left(0,\frac{\sqrt{3}}{L}\right]$ and increasing at $\left(\frac{\sqrt{3}}{L},\frac{2}{L}\right)$. Therefore, the minimum value of $\tau(\gamma)$ occurs at $\frac{\sqrt{3}}{L}$.
\end{proof}

\subsection{Convex Composite Case}
We consider the performance of the PGM under Condition \ref{con2} and assume the PL inequality \eqref{eq:sPL} holds.
We still need the iteration format \eqref{eq:iter}, the interpolation condition for closed, proper and convex functions \eqref{subeq:cov-interpo}, and the PL inequality \eqref{eq:PLdiscrete} mentioned in Section \ref{sect3.1}. However, we replace \eqref{subeq:L-interpo} with the interpolation condition for $L$-smooth convex functions \eqref{eq:Lconv}. 

Following the PEP procedure in Section \ref{sect2}, we provide an explicit linear convergence rate for the PGM under the PL inequality \eqref{eq:sPL}.
\begin{theorem} \label{thm:cov_PL}
Suppose that Condition \ref{con2} holds, $F$ satisfies the PL inequality \eqref{eq:sPL} on $X=\{x:F(x)\leq F(x_1)\}$, and $x_2$ is generated by Algorithm \ref{alg:the PGM}. Then the following holds.
\begin{enumerate}
\item If $\gamma\in(0,\frac{3}{2L}]$, then
$F(x_2) - F_* \leq \frac{1}{2 \gamma \mu + 1}
 (F(x_1) - F_*).$
\item  If $\gamma\in(\frac{3}{2L},\frac{2}{L})$, then
$F(x_2) - F_* \leq \frac{(L\gamma - 1)^2}{(L\gamma - 1)^2 - L \gamma^{2} \mu + 2 \gamma \mu}
 (F(x_1) - F_*). $
 \end{enumerate}
\end{theorem}
\begin{proof}
1. First consider the case of $\gamma \in(0,\frac{3}{2L}]$. Let
$\alpha_1 = \frac2{2\gamma\mu+1}, \alpha_2 = \beta_1 = \frac1{2\gamma\mu+1}, \lambda_1 = \frac{2\gamma\mu}{2\gamma\mu+1}.$
It is obvious that $\alpha_1, \alpha_2, \beta_1$, and $\lambda_1$ are nonnegative. We obtain that
\begin{equation}
\begin{aligned}
F(x_2) & - F_* - \beta_1 (F(x_1) - F_*)-\alpha_1 B_{1,2} -\alpha_2 B_{2,1} - \beta_1 C_{1,2} -\lambda_1 D_2 \\
& = \frac{2L\gamma-3}{2(2\gamma\mu+1)}\|g_1-g_2\|^2 \leq 0,
\end{aligned} \nonumber
\end{equation}
where $B_{1,2}$ and $B_{2,1}$ are defined in \eqref{eq:Lconv}, $C_{1,2}$ is defined in \eqref{subeq:cov-interpo}, 
$D_2$ is defined in \eqref{eq:PLdiscrete}, and the last inequality holds since $\gamma \in(0,\frac{3}{2L}]$. Because the terms after the nonnegative parameters $\alpha_1,\alpha_2,\beta_1$, and $\lambda_1$ are nonpositive, we have
\begin{equation}
F(x_2) - F_* \leq \beta_1 (F(x_1) - F_*) = \frac{1}{2 \gamma \mu + 1} (F(x_1) - F_*). \nonumber
\end{equation}

2. For the case of $\gamma \in(\frac{3}{2L},\frac{2}{L})$, let
$\alpha_1 = \frac{L\gamma-1}{(L\gamma-1)^2-L\gamma^2\mu+2\gamma\mu}, \alpha_2 = \frac{-L^{2}\gamma^{2}+3L\gamma-2}{(L\gamma-1)^2-L\gamma^2\mu+2\gamma\mu},$
$\beta_1 = \frac{(L\gamma - 1)^2}{(L\gamma-1)^2-L\gamma^2\mu+2\gamma\mu}, \lambda_1 = \frac{-L\gamma^2\mu+2\gamma\mu}{(L\gamma-1)^2-L\gamma^2\mu+2\gamma\mu}.$
Similar to the previous case, we obtain that
\begin{equation}
\begin{aligned}
F(x_2) & - F_* - \beta_1 (F(x_1) - F_*)-\alpha_1 B_{1,2} -\alpha_2 B_{2,1}-\beta_1 C_{1,2} -\lambda_1 D_2 \\
& =\frac{-2L\gamma+3}{2L\left(-L\gamma^2\mu+2\gamma\mu + \left(L\gamma-1\right)^2\right)}\|Lg_1\gamma+L\gamma s_2-g_1+g_2\|^2 \leq 0,
\end{aligned} \nonumber
\end{equation}
where $B_{1,2}$ and $B_{2,1}$ are defined in \eqref{eq:Lconv}, $C_{1,2}$ is defined in \eqref{subeq:cov-interpo}, 
$D_2$ is defined in \eqref{eq:PLdiscrete}, and the last inequality holds since $\gamma \in(\frac{3}{2L},\frac{2}{L})$. The remainder of the proof follows a similar approach to the previous case.
\end{proof}

By minimizing the convergence rate in Theorem \ref{thm:cov_PL}, the next proposition gives the ``optimal'' step size for the bound. The proof is similar to that of Proposition \ref{prop:1} and thus omitted.
\begin{proposition} \label{prop:cov_optstepPL}
Suppose that $F$ satisfies the same conditions in Theorem \ref{thm:cov_PL}. Then the ``optimal'' step size for the PGM with the worst-case convergence rate in Theorem \ref{thm:cov_PL} is given by $\frac{3}{2L}$ in $\gamma \in (0,\frac{2}{L})$.
\end{proposition}

\textbf{Comparison.} Now let us compare our convergence results with existing literature.
Recently, Garrigos et al. \cite{garrigos2023convergence}  show that the convergence of the PGM (in our notation) is
\begin{equation} \label{eq:g2023rate1}
F(x_{2})-F^*\leq\frac{1}{1+\gamma\mu(2-\gamma L)}(F(x_1)-F^*).
\end{equation}
Next we show that our rate in Theorem \ref{thm:cov_PL} is better than \eqref{eq:g2023rate1} by considering two intervals.
When $\gamma \in(0,\frac{3}{2L}]$, due to $2-\gamma L \in [\frac{1}{2}, 2)$, we have
\begin{equation}
\frac{1}{1+2\gamma\mu} \leq \frac{1}{1+\gamma\mu(2-\gamma L)} \nonumber
\end{equation}
and thus our bound in Theorem \ref{thm:cov_PL} is tighter.
Next let $\gamma \in(\frac{3}{2L},\frac{2}{L})$. Since $(L\gamma - 1)^2 \in (\frac{1}{4},1)$, we have
\begin{equation}
\gamma\mu(2-\gamma L) = 2\gamma\mu- L\gamma^{2}\mu \leq \frac{2\gamma\mu- L\gamma^{2}\mu}{(L\gamma - 1)^2}, \nonumber
\end{equation}
which further implies
\begin{equation}
\frac{(L\gamma - 1)^2}{(L\gamma - 1)^2 - L \gamma^{2} \mu + 2 \gamma \mu} = \frac{1}{1+\frac{2\gamma\mu- L\gamma^{2}\mu}{(L\gamma - 1)^2}} \leq \frac{1}{1+\gamma\mu(2-\gamma L)}. \nonumber
\end{equation}
Therefore, we provide a tighter convergence rate for $\gamma \in (0, \frac{2}{L})$ than \eqref{eq:g2023rate1}.
In addition, the step size $\frac{1}{L}$ corresponds to the ``optimal'' step size for \eqref{eq:g2023rate1}. When the step size is $\frac{1}{L}$, the contraction rate of \eqref{eq:g2023rate1} is
$\frac{L}{L+\mu}.$
Proposition \ref{prop:cov_optstepPL} shows that $\frac{3}{2L}$ give a rate
$\frac{L}{L+3\mu}$,
which is strictly better than \eqref{eq:g2023rate1}.
An illustration of the convergence rate computed by PEP with the bound \eqref{eq:g2023rate1} is shown in Fig. \ref{fig:1}.
\begin{figure}[ht]
    \centering
    \begin{minipage}{0.45\textwidth}
        \centering
        \includegraphics[width=\linewidth]{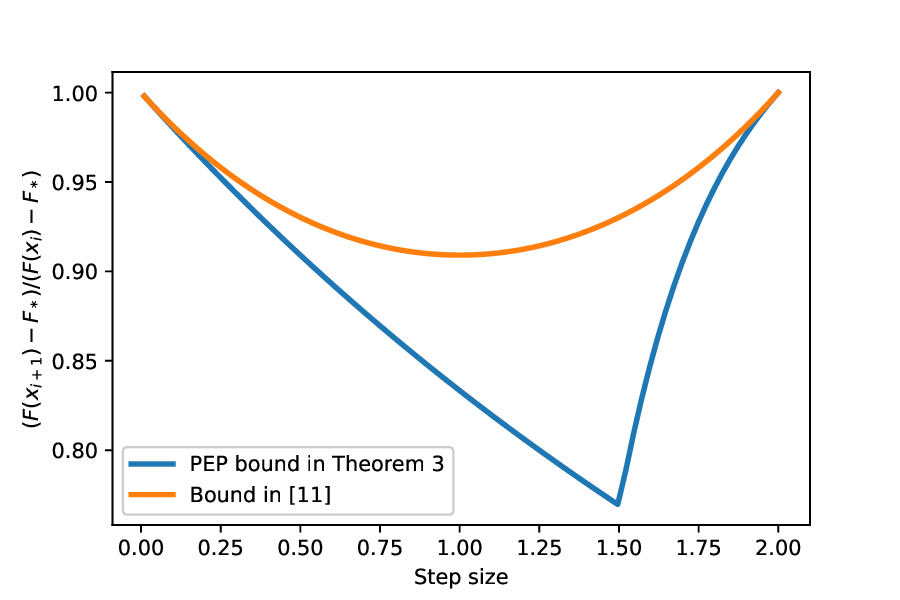}
\caption{Comparison between convergence rates computed by PEP and bound given by \eqref{eq:g2023rate1} in \cite{garrigos2023convergence} ($L=1$, $\mu=0.1$).}
\label{fig:1}
    \end{minipage}
 \hspace{0.05\textwidth}
    \begin{minipage}{0.45\textwidth}
        \centering
        \includegraphics[width=\linewidth]{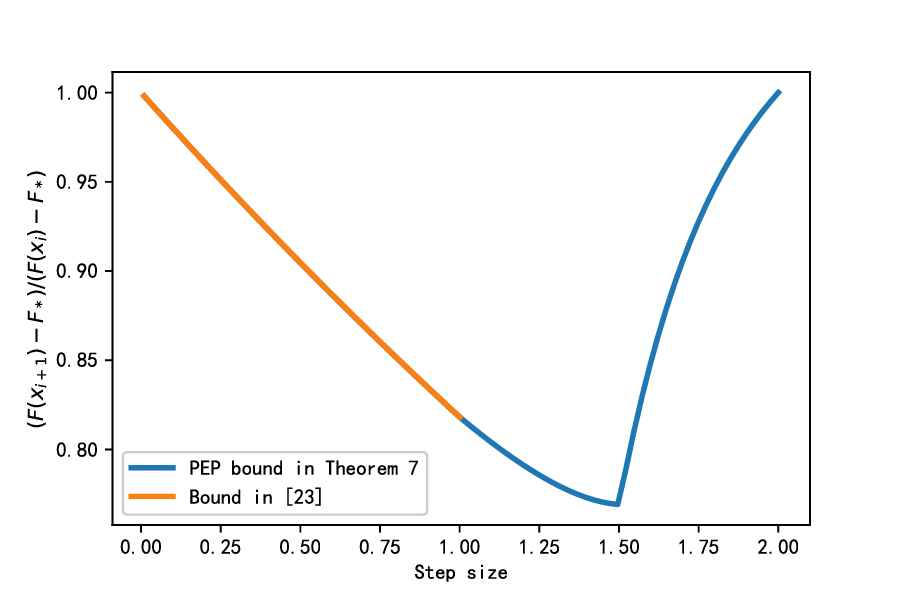}
\caption{Comparison between convergence rates computed by PEP and bound given by \eqref{eq:zhang9} in \cite{zhang2019new} ($L=1$, $\mu=0.1$).}
\label{fig:2}
    \end{minipage}
\end{figure}

\section{Linear Convergence under the RPL Inequality} \label{sect4}
In this section, we study the linear convergence of the PGM under the RPL inequality \eqref{eq:RPL}. We use the performance estimation to obtain the explicit linear convergence rate and the ``optimal'' step size.

\subsection{Nonconvex Composite Case}
First, we consider the performance of the PGM under Condition \ref{con1} and assume the RPL inequality \eqref{eq:RPL} holds.
In addition to the iteration format \eqref{eq:iter}, the interpolation condition for closed, proper and convex functions \eqref{subeq:cov-interpo}, and the interpolation condition for $L$-smooth functions \eqref{eq:Lconv} mentioned in Section \ref{sect2}, we also need to find a discretization of \eqref{eq:RPL} in a Gram-representable form.  If $f$ satisfies the RPL inequality \eqref{eq:sPL}, we can discretize it as
\begin{equation} \label{eq:RPLdiscrete}
E_i := f_i+h_i-F_* - \frac{1}{2\mu} \| \mathcal{G}_i\|^2 \leq 0,
\end{equation}
where $\mathcal{G}_i=\mathcal{G}_{\gamma}(x_i)$.
Furthermore, plug $x_{i+1}=\mathrm{prox}_{\gamma h}(x_i-\gamma\nabla f(x_i))$ and iteration $x_{i+1} = x_i-\gamma(g_i+s_{i+1})$ into the proximal map-based residual $\mathcal{G}_{\gamma}(x_i):=\frac{1}{\gamma}\left(x_i-\mathrm{prox}_{\gamma h}(x_i-\gamma\nabla f(x_i))\right)$,
we have
\begin{equation} \label{eq:GMdiscrete}
\mathcal{G}_{\gamma}(x_i)=\frac{1}{\gamma}(x_i-x_{i+1})=g_i+s_{i+1}.
\end{equation}
Therefore, the RPL inequality \eqref{eq:RPL} implies
\begin{equation}\label{eq:Eprimal}
E'_i := (f_i + h_i)-(f_* + h_*) - \frac{1}{2\mu}(g_i+s_{i+1})^2 \leq 0. 
\end{equation}

Following the PEP procedure in Section \ref{sect2}, we provide an explicit linear convergence rate for the PGM under the RPL inequality \eqref{eq:RPL}.
\begin{theorem} \label{thm:noncov_RPL}
Suppose that Condition \ref{con1} holds, $F$ satisfies the RPL inequality \eqref{eq:RPL} on $X=\{x:F(x)\leq F(x_1)\}$, and $x_2$ is generated by Algorithm \ref{alg:the PGM}. If $\gamma \in\left(0,\frac{2}{L}\right)$, then we have
\begin{equation}
F(x_2) - F_* \leq \frac{L+\mu\left(L\gamma-1\right)^2-\mu}L (F(x_1) - F_*). \nonumber
\end{equation}
\end{theorem}
\begin{proof}

Let $\alpha_1 = \beta_1 = 1$, and $\lambda_1 = \frac{\mu-\mu\left(L\gamma-1\right)^2}{L}$.
By multiplying the inequalities by their respective factors and then summing them, we obtain that
\begin{equation}
\begin{aligned}
F(x_2) & - F_* - \rho (F(x_1) - F_*)-\alpha_1 A_{1,2} -\beta_1 C_{1,2} -\lambda_1 E'_1 \\
& = -\frac{1}{4L}\|Lg_1\gamma+L\gamma s_2-g_1+g_2\|^2 \leq 0.
\end{aligned} \nonumber
\end{equation}
where $A_{1,2}$ is defined in \eqref{subeq:L-interpo}, $C_{1,2}$ is defined in \eqref{subeq:cov-interpo}, and $E'_1$ is defined in \eqref{eq:Eprimal}. Because the terms after the nonnegative parameters $\alpha_1,\beta_1,\lambda_1$ are nonpositive, we have the desired result.
\end{proof}

The following proposition provides the ``optimal'' step size by minimizing the worst-case convergence rate in Theorem \ref{thm:noncov_RPL}, whose proof is similar to that of Proposition \ref{prop:1} and thus omitted.
\begin{proposition}
Suppose that $F$ satisfies the same conditions in Theorem \ref{thm:noncov_RPL}. Then the optimal step size for the PGM with the worst-case convergence rate in Theorem \ref{thm:noncov_RPL} is given by $\frac{1}{L}$ in $\gamma \in (0,\frac{2}{L})$.
\end{proposition}

Under the RPL inequality \eqref{eq:RPL}, we provide the first explicit linear convergence rate of the PGM. Additionally, we show that the ``optimal'' step size is $\frac{1}{L}$ in this setting.

\subsection{Convex Composite Case}
Next, we consider the performance of the PGM under Condition \ref{con2} and assume the RPL inequality \eqref{eq:RPL} holds. Zhang and Zhang \cite{zhang2019new} proved a linear convergence rate for the PGM under the RPL inequality. By using PEP, we extend the step size interval to $(0,\frac{2}{L})$ and find a better convergence rate than that in \cite{zhang2019new}. 
We use the interpolation conditions of $\mathcal{F}_{0,\infty}(\mathbb{R}^d)$ in \eqref{subeq:cov-interpo} and $\mathcal{F}_{0,L}(\mathbb{R}^d)$ in \eqref{eq:Lconv}, respectively, and the discrete RPL inequality \eqref{eq:RPLdiscrete}.
Now we are ready to present our result.
\begin{theorem} \label{thm:cov_RPL}
Suppose that Condition \ref{con2} holds, $F$ satisfies the RPL inequality \eqref{eq:RPL} on $X=\{x:F(x)\leq F(x_1)\}$, and $x_2$
 is generated by Algorithm \ref{alg:the PGM}. Then the following holds.
\begin{enumerate}
  \item If $\gamma\in(0,\frac{1}{L}]$, then
$F(x_2) - F_* \leq \frac{1-\gamma \mu}{1+\gamma \mu}
 (F(x_1) - F_*).$
\item If $\gamma\in(\frac{1}{L},\frac{3}{2L}]$, then
$F(x_2) - F_* \leq \frac{-2L\gamma^2\mu+L\gamma+3\gamma\mu-2}{L\gamma-\gamma\mu-2}
 (F(x_1) - F_*).$
\item If $\gamma\in(\frac{3}{2L},\frac{2}{L})$, then
$F(x_2) - F_* \leq \frac{(L\gamma - 1)^2}{(L\gamma - 1)^2 - L \gamma^{2} \mu + 2 \gamma \mu}
 (F(x_1) - F_*).$
 \end{enumerate}
\end{theorem}
\begin{proof}
1. First consider the case of $\gamma \in(0,\frac{1}{L}]$. Let
$\alpha_1 = \beta_1 = \frac{1}{1+\gamma\mu}, \lambda_1 = \lambda_2 = \frac{\gamma\mu}{1+\gamma\mu}.$
It is obvious that $\alpha_1, \beta_1, \lambda_1$, and $\lambda_2$ are nonnegative. Note that
\begin{equation} \label{eq:lemma}
\|g_{i+1}+s_{i+1}\|\geq d(0,\partial F(x_{i+1}))\geq\|\mathcal{G}_{i+1}\| \nonumber
\end{equation}
from Lemma \ref{lemma}, and $\|\mathcal{G}_i\|^2 = \|g_i+s_{i+1}\|^2$ in \eqref{eq:GMdiscrete}. We then have
\begin{equation}
\begin{aligned}
F(x_2) &- F_* - (1-\lambda_1-\lambda_2) (F(x_1) - F_*)-\alpha_1 B_{1,2} -\beta_1 C_{1,2} -\lambda_1 E_1 -\lambda_2 E_2 \\
\leq & -\alpha_1(\langle g_{2},x_{1}-x_{2}\rangle + \frac{1}{2L}\|g_{1}-g_{2}\|^{2}) -\beta_1\langle s_{2},x_{1}-x_{2}\rangle+\frac{\lambda_1}{2u}\|g_{1}+s_{2}\|^{2} \\
& +\frac{\lambda_2}{2u}\|g_{2}+s_{2}\|^{2} = (L\gamma-1)\|g_1-g_2\|^2 \leq 0,
\end{aligned} \nonumber
\end{equation}
where $B_{1,2}$ is defined in \eqref{eq:Lconv}, $C_{1,2}$ is defined in \eqref{subeq:cov-interpo}, $E_1$ and $E_2$ are defined in \eqref{eq:RPLdiscrete} and the last inequality holds since $\gamma \in(0,\frac{1}{L}]$. Because the terms after the nonnegative parameters $\alpha_1,\beta_1,\lambda_1$ and $\lambda_2$ are nonpositive, we have
\begin{equation}
F(x_2) - F_* \leq (1-\lambda_1-\lambda_2) (F(x_1) - F_*) = \frac{1-\gamma \mu}{1+\gamma \mu}
 (F(x_1) - F_*). \nonumber
\end{equation}

2. For the case of $\gamma \in(\frac{1}{L}, \frac{3}{2L}]$, let
$
\alpha_1 = \frac1{-L\gamma+\gamma\mu+2}, \quad \alpha_2 = \frac{L\gamma-1}{-L\gamma+\gamma\mu+2}, \quad \beta_1 = \frac{-L\gamma+2}{-L\gamma+\gamma\mu+2},
\lambda_1 = \frac{\gamma\mu\left(-2L\gamma+3\right)}{-L\gamma+\gamma\mu+2}, \quad \lambda_2 = \frac{\gamma\mu}{-L\gamma+\gamma\mu+2}. $
We can also check that $\alpha_1, \alpha_2, \beta_1, \lambda_1$, and $\lambda_2$ are nonnegative. Using the same discussions on Lemma \ref{lemma} and \eqref{eq:GMdiscrete}, we have
\begin{equation}
\begin{aligned}
F(x_2) & - F_* - (1-\lambda_1-\lambda_2) (F(x_1) - F_*)-\alpha_1 B_{1,2} -\alpha_2 B_{2,1} - \beta_1 C_{1,2} -\lambda_1 E_1 - \lambda_2 E_2 \\
\leq & -\alpha_1(\langle g_{2},x_{1}-x_{2}\rangle + \frac{1}{2L}\|g_{1}-g_{2}\|^{2})-\alpha_2(\langle g_{1},x_{2}-x_{1}\rangle + \frac{1}{2L}\|g_{2}-g_{1}\|^{2}) \\
& -\beta_1\langle s_{2},x_{1}-x_{2}\rangle + \frac{\lambda_1}{2u}\|g_{1}+s_{2}\|^{2}+\frac{\lambda_2}{2u}\|g_{2}+s_{2}\|^{2} = 0.
\end{aligned} \nonumber
\end{equation}
The remainder of the proof follows a similar approach to the first case.

3. For the case of $\gamma \in(\frac{3}{2L}, \frac{2}{L}]$, let
$\alpha_1 = \frac{L\gamma-1}{(L\gamma-1)^2-L\gamma^2\mu+2\gamma\mu}, \alpha_2 = \frac{-L^{2}\gamma^{2}+3L\gamma-2}{(L\gamma-1)^2-L\gamma^2\mu+2\gamma\mu},$
$\beta_1 = \frac{(L\gamma - 1)^2}{(L\gamma-1)^2-L\gamma^2\mu+2\gamma\mu}, \lambda_1 = \frac{-L\gamma^2\mu+2\gamma\mu}{(L\gamma-1)^2-L\gamma^2\mu+2\gamma\mu}.$
We can also check that $\alpha_1, \alpha_2, \beta_1$, and $\lambda_1$ are nonnegative. Therefore, we have
\begin{equation}
\begin{aligned}
F(x_2) & - F_* - \beta_1 (F(x_1) - F_*)-\alpha_1 B_{1,2} -\alpha_2 B_{2,1}-\beta_1 C_{1,2} - \lambda_1 E_1 \\
\leq & -\alpha_1(\langle g_{2},x_{1}-x_{2}\rangle + \frac{1}{2L}\|g_{1}-g_{2}\|^{2})-\alpha_2(\langle g_{1},x_{2}-x_{1}\rangle + \frac{1}{2L}\|g_{2}-g_{1}\|^{2})  \\
& -\beta_1\langle s_{2},x_{1}-x_{2}\rangle +\frac{\lambda_1}{2u}\|g_{1}+s_{2}\|^{2} \\
=& \frac{-2L\gamma+3}{2L\left(-L\gamma^2\mu+2\gamma\mu+\left(L\gamma-1\right)^2\right)}\|Lg_1\gamma+L\gamma s_2-g_1+g_2\|^2 \leq 0,
\end{aligned} \nonumber
\end{equation}
where the last inequality holds since $\gamma \in(\frac{3}{2L}, \frac{2}{L}]$. The remainder of the proof follows a similar approach to the first case.
\end{proof}

By minimizing the convergence rate in Theorem \ref{thm:cov_RPL}, the next proposition gives the ``optimal'' step size for the bound, which is shown to be in the interval  $(\frac{1}{L},\frac{3}{2L}]$. 
\begin{proposition} \label{prop:cov_optstepRPL}
Suppose that $F$ satisfies the same conditions in Theorem \ref{thm:cov_RPL}. Then the optimal step size for the bound in Theorem \ref{thm:cov_RPL} is either $\frac{3}{2L}$ if $\mu\le \frac{L}{9}$, or $\frac{2}{\sqrt L(\sqrt L+\sqrt\mu)}$ if $\frac{L}{9}<\mu<L$.
\end{proposition} 
\begin{proof}
Let
\begin{equation}
\tau(\gamma) =\begin{cases}\frac{1-\gamma \mu}{1+ \gamma \mu} & \gamma \in\left(0,\frac{1}{L}\right]\\
\frac{-2L\gamma^2\mu+L\gamma+3\gamma\mu-2}{L\gamma-\gamma\mu-2} & \gamma \in\left(\frac{1}{L},\frac{3}{2L}\right]\\
\frac{(L\gamma - 1)^2}{(L\gamma - 1)^2 - L \gamma^{2} \mu + 2 \gamma \mu}&\gamma \in(\frac{3}{2L},\frac{2}{L}).\end{cases} \nonumber
\end{equation}
First note that $\tau'(\gamma) < 0$ with $\gamma \in\left(0,\frac{1}{L}\right)$. So  $\tau(\gamma)$ is decreasing in $\left(0,\frac{1}{L}\right]$.

Second, we consider $\gamma \in\left(\frac{1}{L},\frac{3}{2L}\right]$. Through some algebra, we can show that $\tau''(\gamma) \geq 0$, then $\tau(\gamma)$ is convex on the given interval. Note that \[\tau'(\gamma) = (2\mu(-L^2\gamma^2 + L\gamma^2\mu + 4L\gamma - 4))/(-L\gamma + \gamma\mu + 2)^2.\]
When $\mu\le \frac{L}{9}$, we have $\tau'(\gamma)\le 0$ in $(\frac{1}{L},\frac{3}{2L})$, and thus $\tau(\gamma)$ is nonincreasing in $(\frac{1}{L},\frac{3}{2L}]$, which implies the minimizer in $(\frac{1}{L},\frac{3}{2L}]$ is the endpoint $\frac{3}{2L}$.
When $\frac{L}{9}<\mu\le L$,  we obtain that $\tau'(\frac{2}{\sqrt L(\sqrt L+\sqrt\mu)})=0$ with $\frac{2}{\sqrt L(\sqrt L+\sqrt\mu)}\in \left(\frac{1}{L},\frac{3}{2L}\right)$, and thus $\frac{2}{\sqrt L(\sqrt L+\sqrt\mu)}$ is the minimizer of $\tau(\gamma)$ in $\left(\frac{1}{L},\frac{3}{2L}\right]$.

Third, $\tau'(\gamma) > 0$ for $\gamma \in\left(\frac{3}{2L},\frac{2}{L}\right)$, so  $\tau(\gamma)$ is increasing in $\left(\frac{3}{2L},\frac{2}{L}\right)$.
 
Summarizing the above cases and noting that $\gamma(\tau)$ is continuous at $\frac{1}{L}$ and $\frac{3}{2L}$, we obtain the desired results.
\end{proof}

\textbf{Comparison.} Zhang and Zhang \cite{zhang2019new} showed a linear convergence rate for the PGM under the RPL inequality with $\gamma \in (0,\frac{1}{L}]$:
\begin{equation} \label{eq:zhang9}
F(x_{k+1}) - F_* \leq \frac{1-\gamma \mu}{1+\gamma \mu}
 (F(x_k) - F_*).
\end{equation}
A comparison of the convergence rate computed by PEP with the bound \eqref{eq:zhang9} is shown in Fig. \ref{fig:2}. The convergence rate in \cite{zhang2019new} is obtained  using a refined descent lemma on the interval $(0,\frac{1}{L}]$, while we extend the interval to $(0,\frac{2}{L})$.
Combined with Proposition \ref{prop:cov_optstepRPL}, we find that by choosing a longer step size, one can achieve a better convergence rate than \eqref{eq:zhang9}.

\begin{remark}
From Theorems \ref{thm:cov_PL} and \ref{thm:cov_RPL}, we can observe that for convex composite functions, whether satisfying the PL or RPL inequality, the convergence rates are the same when the step size $\gamma \in(\frac{3}{2L},\frac{2}{L})$ but different in other intervals of step sizes. 
We remark that the main difference between the analysis of the two cases is that \eqref{eq:PLdiscrete} is used for the PL case and \eqref{eq:Eprimal} is used for the RPL case.
\end{remark}

\section{Numerical Experiments}\label{sect5}
In this section, we compare the performance of different step size strategies for the PGM through numerical experiments on various practical problems.  

\subsection{Numerical comparisons under the PL inequality}
First, we consider a nonconvex composite problem, the sparse robust linear regression (SRLR) problem, which has been studied in prior work, such as Ochs et al. \cite{ochs2014ipiano}, Zhou et al. \cite{zhou2020proximal}, and Khanh et al. \cite{khanh2025inexact}. The problem is formulated as:
\begin{equation}
\min_{x\in\mathbb{R}^d} \sum_{k = 1}^{n} \log \left((Ax - b)_k^2 + 1\right) + \lambda\|x\|_1, \nonumber
\end{equation}
where $A\in\mathbb{R}^{n\times d}$, $b\in\mathbb{R}^n$, $\lambda>0$ is the regularization parameter, and $\|\cdot\|_1$ denotes the $\ell_1$ norm. We generate $A \in \mathbb{R}^{n \times d}$ and $b \in \mathbb{R}^n$ with independent and identically distributed (i.i.d.) standard normal entries, where $n = 200$ and $d = 20$, and set the regularization parameter to $\lambda = 0.1$. 
Fig. \ref{fig:3} shows the results of using PGM to solve the SRLR problem with different step sizes.
Particularly, we report the suboptimality gap, $F(x_i) - F(\hat{x})$, where $\hat{x}$ denotes the suboptimal solution obtained by running PGM for a large number of iterations with high tolerance. 

As shown in Fig. \ref{fig:3}, PGM initially exhibits sublinear convergence, and transitions to linear convergence as it approaches a limit point. Since the objective function is the sum of a logarithmic function composed with a polynomial and the $\ell_1$ norm, it possesses an o-minimal structure. According to Bolte et al.~\cite{bolte2007lojasiewicz}, functions with o-minimal structure satisfy the KL inequality. In particular, we verified that the Hessian of the smooth part of the objective is positive definite near the point of convergence, implying local strong convexity. This suggests that the KL exponent tends to $\frac{1}{2}$, confirming that the PL inequality holds in a neighborhood of the point of convergence. This observation explains the linear convergence behavior near the optimum. We further observe that the PGM with the ``optimal" step size $\frac{\sqrt{3}}{L}$, as stated in Proposition~\ref{prop:1}, outperforms other step sizes.

\begin{figure}[ht]
    \centering
    \begin{minipage}{0.45\textwidth}
        \centering
        \includegraphics[width=\linewidth]{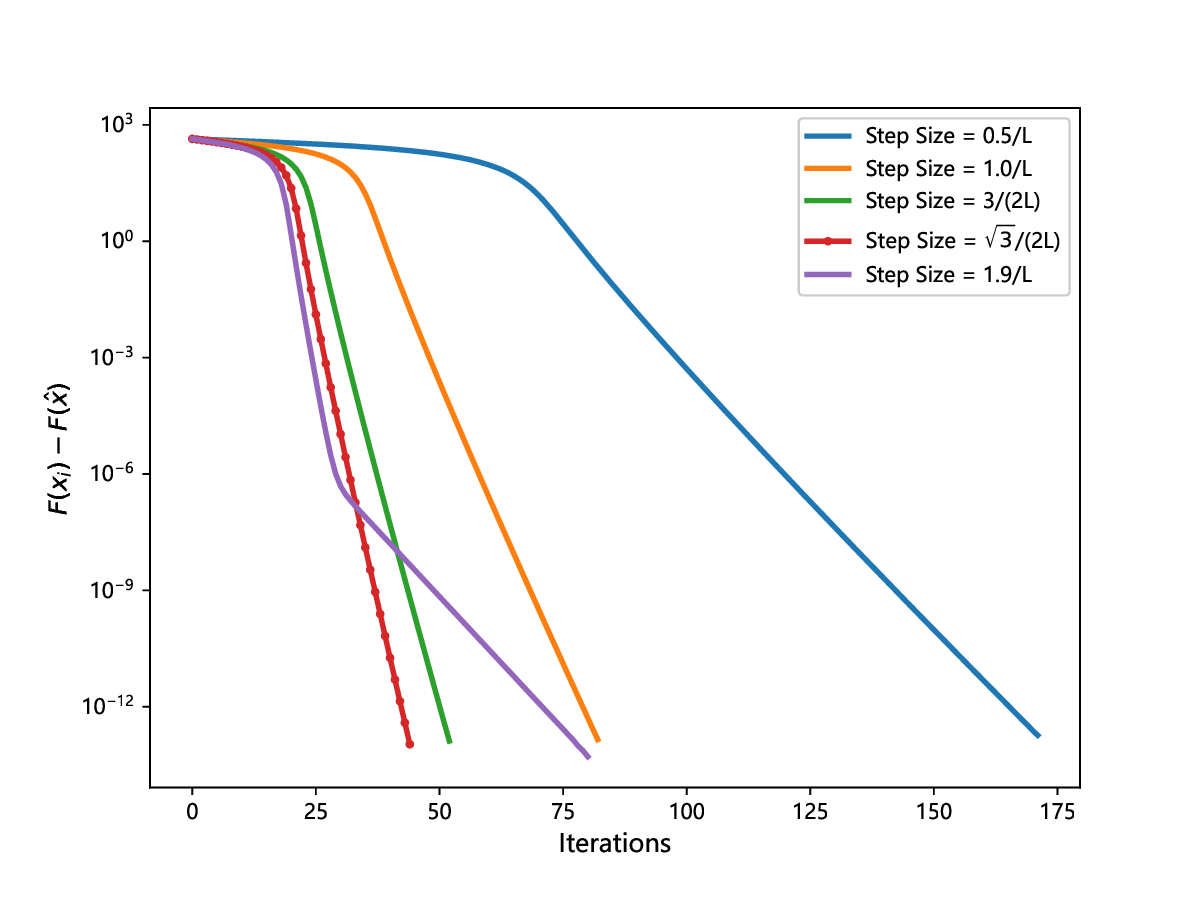}
\caption{Comparison on the SRLR problem.}
\label{fig:3}
    \end{minipage}
 \hspace{0.05\textwidth}
    \begin{minipage}{0.45\textwidth}
        \centering
        \includegraphics[width=\linewidth]{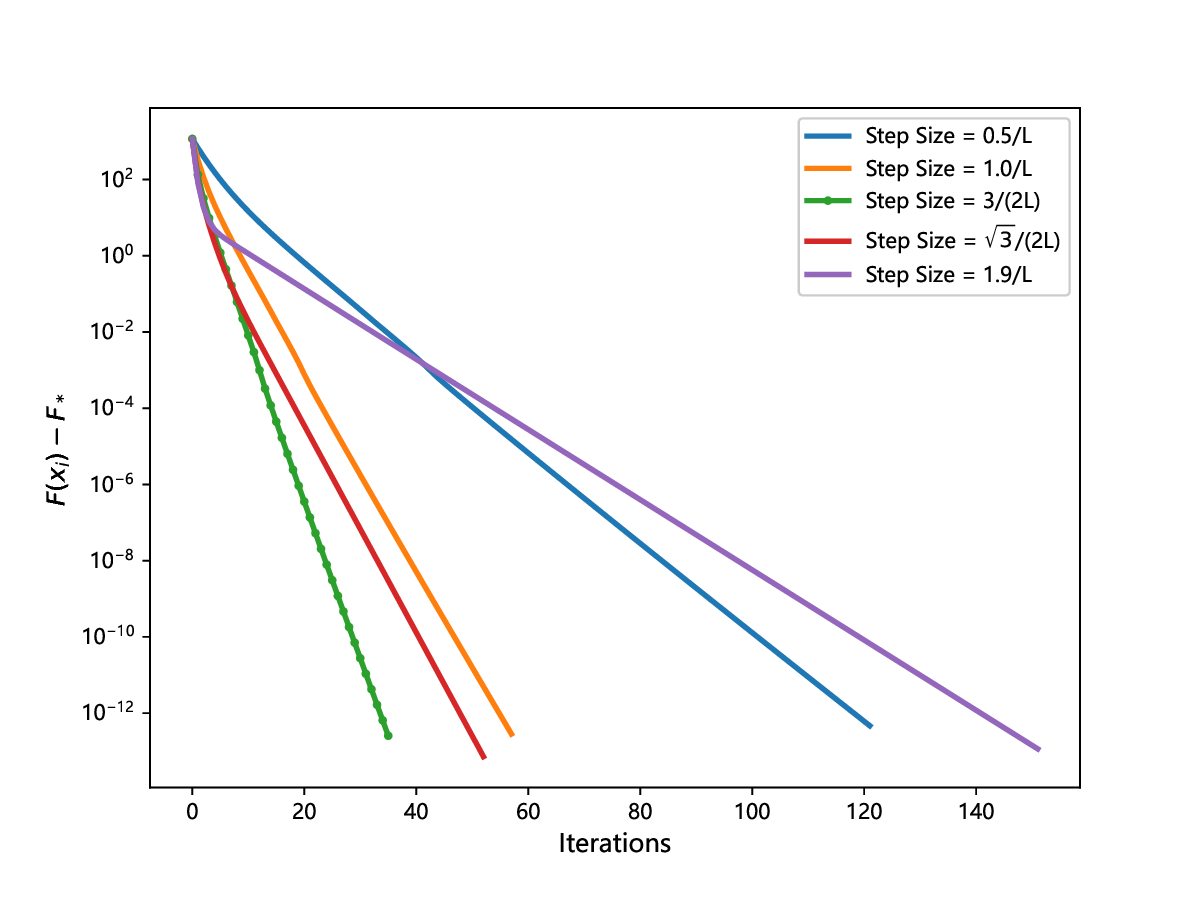}
\caption{Comparison on the LASSO problem.}
\label{fig:4}
    \end{minipage}
\end{figure}

Next, we consider a convex composite problem, the LASSO problem \cite{tibshirani1996regression}:
\begin{equation} \label{eq:LASSO}
\min_{x\in\mathbb{R}^d} \frac{1}{2}\|Ax-b\|^2 + \lambda\|x\|_1, \nonumber
\end{equation}
where $A\in\mathbb{R}^{n\times d}$, $b\in\mathbb{R}^n$, and $\lambda>0$. We generate $A$ and $b$ in the same way as in the SRLR problem and set $\lambda=0.1$. It is known that the objective function of the LASSO problem satisfies the KL inequality with exponent $\frac{1}{2}$ \cite{li2018calculus}, which coincides with the PL inequality discussed in this paper. 
The results of using PGM with different step sizes are shown in Fig. \ref{fig:4}. In particular, we report the optimality in the objective function, $F(x_i) - F_*$, with $F_*$ being the optimal solution computed by the solver CVX.  We observe that the PGM using the ``optimal" step size $\frac{3}{2L}$, as given in Proposition~\ref{prop:cov_optstepPL}, converges faster than other step sizes, including the widely used choice $\frac{1}{L}$. This illustrates the effectiveness of our theoretical step size in improving the convergence rate.

\subsection{Numerical comparisons under the RPL inequality}
Finally, we consider numerical comparisons under the RPL inequality.
However, we have not yet identified meaningful practical nonconvex composite problems that satisfy the RPL inequality.
As a result, we do not conduct numerical experiments for the nonconvex case. Instead, we focus on the convex setting and consider the elastic net problem introduced by Zou and Hastie~\cite{zou2005regularization}:
\begin{equation}
\min_{x\in\mathbb{R}^d} \frac{1}{2}\|Ax-b\|^2 + \lambda\|x\|_1+\frac{\delta}{2}\|x\|^2, \nonumber
\end{equation}
where $A\in\mathbb{R}^{n\times d}$, $b\in\mathbb{R}^n$, $\lambda>0$. We generate $A$ and $b$ in the same way as in the SRLR problem and set $\lambda=0.1$. This problem can be viewed as a combination of a smooth convex term $\frac{1}{2} \|Ax - b\|^2$ and a strongly convex regularizer $\lambda \|x\|_1 + \frac{\delta}{2} \|x\|^2$. The objective function is strongly convex and thus satisfies the RPL inequality, and $\mu = \lambda_{\min}(A^T A)+ \delta$.

\begin{figure}[ht]
    \centering
    \begin{minipage}{0.45\textwidth}
        \centering
        \includegraphics[width=\linewidth]{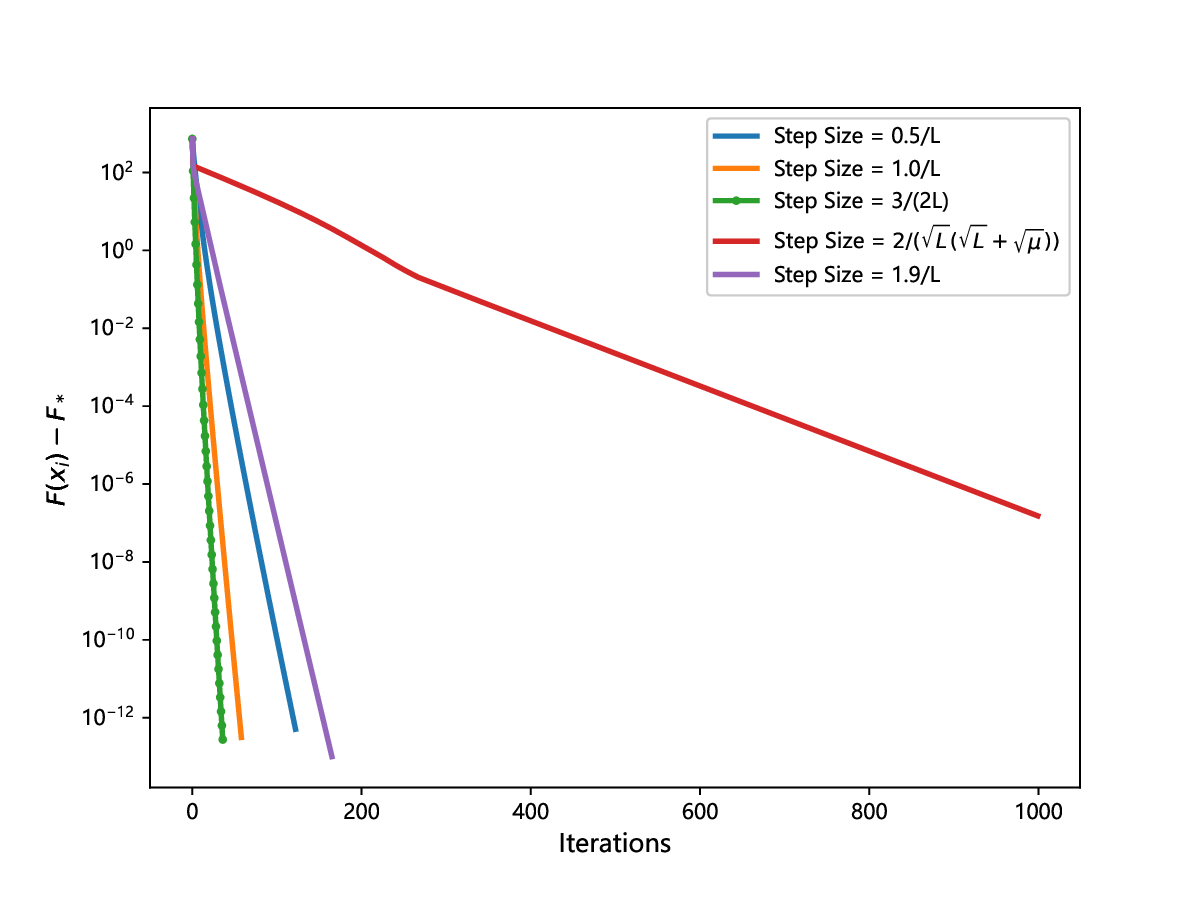}
\caption{Comparison on the elastic net problem ($\mu\le L/9$).}
\label{fig:61}
    \end{minipage}
 \hspace{0.05\textwidth}
    \begin{minipage}{0.45\textwidth}
        \centering
        \includegraphics[width=\linewidth]{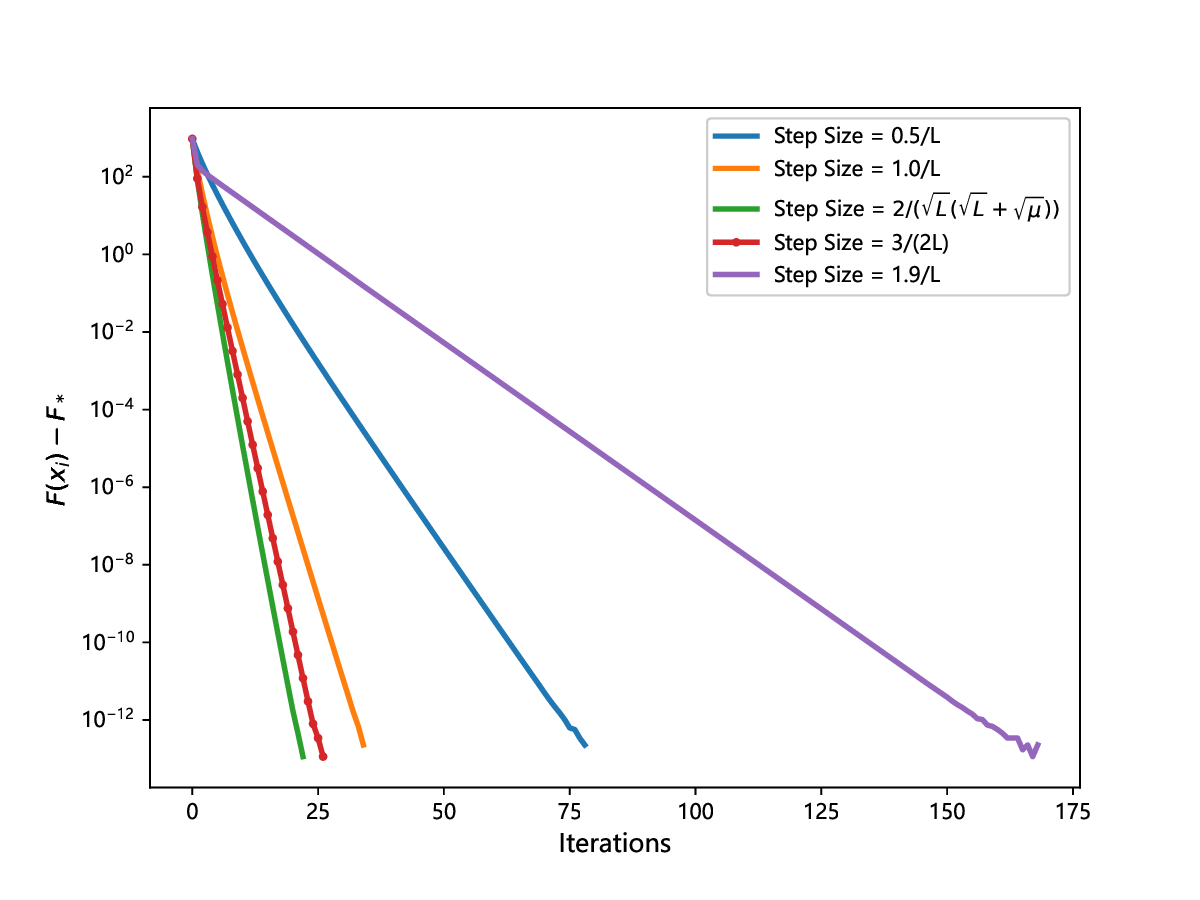}
\caption{Comparison  on the elastic net problem ($L/9<\mu<L$).}
\label{fig:62}
    \end{minipage}
\end{figure}

We conduct experiments using the same setup as in the LASSO problem and compare the performance of PGM with different step sizes. Note that the ``optimal'' step size is either $\frac{3}{2L}$ if $\mu\le \frac{L}{9}$, or $\frac{2}{\sqrt L(\sqrt L+\sqrt\mu)}$ if $\frac{L}{9}<\mu<L$ in Proposition \ref{prop:cov_optstepRPL}. We set $\delta = 10^{-2}$ in Fig. \ref{fig:61} to correspond to the case where $\mu \le \frac{L}{9}$, and $\delta = 100$ in Fig. \ref{fig:62} to correspond to the case where $\frac{L}{9} < \mu < L$. In particular, we computed that the step size $\gamma \approx \frac{1.991}{L}$ is close to $\frac{2}{L}$ when $\delta = 10^{-2}$, and $\gamma \approx \frac{1.404}{L} < \frac{3}{2L}$ when $\delta=100$. The results in both figures show that PGM with the ``optimal'' step size converges faster than with other step sizes, including the commonly used choice $\frac{1}{L}$.

\section{Conclusion}\label{sect6}
In this paper, we provide explicit convergence rates of the PGM  applied to composite functions using the PEP framework, focusing on two classes of the PL inequality. For nonconvex composite functions satisfying either the PL or RPL inequality, we present the first explicit linear convergence rate. Additionally, we derive tighter bounds on the linear convergence rate of PGM for minimizing convex composite functions under the PL or RPL inequality. Finally, we conduct numerical experiments to show that the theoretically ``optimal'' step sizes derived in this paper achieve better convergence rates than other step sizes.

We highlight that the necessary and sufficient interpolation conditions for functions satisfying the PL inequality remain unknown. Addressing this gap would enable the derivation of tight convergence rates for this class of functions, which we leave as future work.\\

\noindent\textbf{Data Availability} The paper contains no data.
\section*{Declarations}
\textbf{Conflict of interest} The authors have no confict of interest to disclose.



%
%





\bibliography{main.bib}

\end{document}